\documentclass{amsart}
\usepackage{amsmath,amssymb,hyperref}

\newtheorem{thm}{Theorem}[section]
\newtheorem{lemma}[thm]{Lemma}
\newtheorem{claim}{Claim}[thm]

\newtheorem{theorem}[thm]{Theorem}
\newtheorem{corollary}[thm]{Corollary}
\newtheorem{question}[thm]{Question}
\newtheorem{conjecture}[thm]{Conjecture}
\newtheorem{proposition}[thm]{Proposition}
\newtheorem{observation}[thm]{Observation}

\newtheorem*{thma}{Theorem A}
\newtheorem*{thmb}{Theorem B}
\newtheorem*{thmc}{Theorem C}

\theoremstyle{definition}
\newtheorem{defn}[thm]{Definition}
\newtheorem{definition}[thm]{Definition}

\DeclareMathOperator{\otp}{otp}
\DeclareMathOperator{\cf}{cf}

\def\s{\subseteq}

\newcommand{\azero}{{\aleph_0}}
\newcommand{\aone}{{\aleph_1}}
\newcommand{\omo}{{\omega_1}}
\newcommand{\atwo}{{\aleph_2}}
\newcommand{\omt}{{\omega_2}}
\newcommand{\po}{\mathbb{P}}

\newcommand{\func}{\rightarrow}

\DeclareMathOperator{\cc}{CC}

\newcommand{\omobaire}{{}^\omo\omo}
\newcommand{\gch}{GCH}

\newcommand{\omthr}{{\omega_3}}

\newcommand{\fsp}[2]{{}^{#1}{#2}}
\newcommand{\seq}[1]{\langle #1 \rangle}

\newcommand{\gu}{{\mathsf{L}}}
\usepackage{enumitem}
\setlist[enumerate,1]{label={(\roman*)}}
\DeclareMathOperator{\card}{Card}

\newenvironment{why}[1][Proof]{\proof[#1]\mbox{}}{\endproof}
\DeclareMathOperator{\cl}{cl}
\DeclareMathOperator{\cov}{cov}
\DeclareMathOperator{\meet}{m}

\DeclareMathOperator{\cof}{cof}
\DeclareMathOperator{\im}{im}
\DeclareMathOperator{\pp}{pp}
\DeclareMathOperator{\ecf}{ecf}


\title{On strong chains of sets and functions}
\date{}
\author{Tanmay C.\ Inamdar}
\address{Department of Mathematics, Bar-Ilan University, Ramat-Gan 5290002, Israel.}

\begin{document}
	\maketitle
	\begin{abstract} 
		Shelah has shown that there are no chains of length $\omega_3$ increasing modulo finite in ${}^{\omega_2}\omega_2$. We improve this result to sets. That is, we show that there are no chains of length $\omega_3$ in $[\omega_2]^{\aleph_2}$ increasing modulo finite. This contrasts with results of Koszmider who has shown that there are, consistently, chains of length $\omega_2$ increasing modulo finite in $[\omega_1]^{\aleph_1}$ as well as in ${}^{\omega_1}\omega_1$. More generally, we study the depth of function spaces ${}^\kappa\mu$ quotiented by the ideal $[\kappa]^{< \theta}$ where $\theta< \kappa$ are infinite cardinals.
	\end{abstract}
	\section{Introduction}
	Throughout this paper $\theta, \kappa, \lambda$ are infinite cardinals with $\theta \leq \kappa$ and $\mu$ is a (possibly finite) cardinal.
	
	An important part of set theory has been the study of $\fsp{\omega}{\omega}$ and $[\omega]^\omega$ quotiented by the ideal $[\omega]^{<\omega}$ of finite subsets of the natural numbers, the area of cardinal characteristics of the continuum (see for example \cite{barjud}). Another area, recently receiving more attention than previously, is a similar study performed for an arbitrary infinite cardinal $\kappa$, perhaps with the addition of extra cardinal arithmetic assumptions. That is, the study of $\fsp{\kappa}{\kappa}$ and $[\kappa]^\kappa$ quotiented by the ideal $[\kappa]^{<\kappa}$, the area of higher cardinal invariants(see for example \cite{kls} and its extensive bibliography). Another natural ideal to consider for $\kappa$ an uncountable regular cardinal is the ideal of non-stationary subsets of $\kappa$ (see for example \cite{galvinhajnal}). 
	
	In this paper we focus on ideals of the form $[\kappa]^{< \theta}$ where $\theta$ and $\kappa$ are infinite cardinals, and $\theta$ is strictly less than $\kappa$. Before we describe our results, let us mention some other results in the literature on this topic.
	
	Baumgartner in his thesis proved the following.
	\begin{theorem}[\cite{baumgartner}] \label{baumgartner}
		Assume $\gch$. Let $\theta\leq \kappa\leq\lambda$ be infinite cardinals with $\theta$ regular. Then there is a cardinal-preserving forcing $\po$ such that in the generic extension by $\po$ we have a $(<\theta)$-almost disjoint family $\{A_\alpha \mid \alpha < \lambda\}$ contained in $[\kappa]^\kappa$: for every $\alpha < \beta < \lambda$, $|A_\alpha \cap A_\beta| < \theta$. 
	\end{theorem}
	Baumgartner's proof in particular shows that it is consistent that there there are arbitrarily large families in $[\omo]^{\aone}$ and even in $[\omt]^{\atwo}$ which are $(<\azero)$-almost disjoint, that is, almost disjoint modulo finite. 
	
	Using the Erd{\H o}s-Rado-Kurepa Theorem (also sometimes called the Erd{\H o}s-Rado Theorem), we see that for example, a $\lambda$-sized $(< \azero)$-almost disjoint family in $[\omo]^\aone$ implies that $2^\azero \geq \lambda$. In \cite{baushe} it is also shown that a generic extension obtained by adding arbitrarily many Cohen reals to a model of $2^\azero = \aone$  does not contain an $(<\azero)$-almost disjoint family in $[\omo]^\aone$ of size $\atwo$.
	
	Zapletal in his thesis (see \cite{zapletal}) proved a related result about such `strongly' almost disjoint families with some additional properties, but we do not go into the details here. We do however mention the motivation for Zapletal as it will allow us to introduce the next result we shall discuss.
	
	The result of Baumgartner's suggests a general investigation of the possible quasiorders contained within spaces of the form $[\kappa]^\kappa$ or ${}^\kappa\kappa$ modulo the ideal $[\kappa]^{<\theta}$. Hechler's Theorem (see \cite{hechler}) solves this question completely for the usual Baire space and see \cite{cummingsshelah} and \cite{rinot39} for similar results for the other spaces we have mentioned.
	
	Then Baumgartner's result can be interpreted as a result about the `width' of such spaces, so Hajnal and Szentmikl\'{o}ssy asked the next natural question, what about their `depth'? In order to make things precise we need to introduce some notation which we shall use throughout this article.
	
	Let $X, Y$ be sets of ordinals, $X$ infinite, and let $\theta$ be an infinite cardinal such that $\theta < |X|$. For $f, g \in \fsp{X}{Y}$, let $f \ll_\theta g$ denote that $|\{\xi \in X \mid f(\xi) \geq g(\xi)\}| < \theta$. Note that in this case $|\{\xi \in X \mid f(\xi) < g(\xi)\}| = |X|$.
	
	The question of Hajnal and Szentmikl\'{o}ssy which motivated Zapletal was the following.
	\begin{question}[pg.\ 435 of \cite{hajnal}]\label{hajnalquestion}
		Is it consistent that there is a chain in $({}^{\omega_1}\omega_1, \ll_{\aleph_0})$ of length $\omega_2$? That is, is it consistent that there is a sequence $\langle f_\alpha \mid \alpha< \omega_2\rangle$ of elements in ${}^{\omega_1}\omega_1$ such that $\alpha< \beta< \omega_2$ implies that $f_\alpha \ll_{\aleph_0} f_\beta$?
	\end{question}
	Note that while we use the term `chain' above and throughout the paper as it is more intuitively clear, the term `depth' is also used for a related concept, for example in \cite{shelah}. To translate between the two, the depth is the supremum of the length of well-ordered chains in a partially ordered set. In particular, the above question could be phrased as asking if it is consistent that the depth of the partial order $({}^{\omega_1}\omega_1, \ll_{\aleph_0})$ is greater than $\omega_2$. We shall however always use the term `chain' but since we shall only be interested in well-ordered chains, it is straightforward to translate between our terminology and Shelah's. 
	
	Hajnal also mentions another related question.
	\begin{question}[pg.\ 435 of \cite{hajnal}]\label{hajnalsetsquestion}
		Is it consistent that there is a chain of sets $\left\langle X_\alpha \mid \alpha < \omt\right\rangle$ in $[\omo]^\aone$ such that for $\alpha < \beta< \omt$, $X_\alpha \setminus X_\beta$ is finite, and $|X_\beta \setminus X_\alpha| = \aone$?
	\end{question}
	As pointed out by Hajnal, given a family as above, $\{X_{\beta +1} \setminus X_\beta \mid \beta < \omt\}$ is an almost disjoint modulo finite family in $[\omo]^\aone$ as in Baumgartner's Theorem~\ref{baumgartner}.
	
	In order to see the relation between the two questions, consider the following weakening of $\ll_\theta$: for $f, g \in \fsp{X}{Y}$, let $f <_\theta g$ denote that $|\{\xi \in X \mid f(\xi) > g(\xi)\}| < \theta$ and $|\{\xi \in X \mid f(\xi) < g(\xi)\}| = |X|$.  It is clear then that $f \ll_\theta g$ implies that $f <_\theta g$, and that Question~\ref{hajnalsetsquestion} asks about the consistent existence of an $\omega_2$-chain in $({}^{\omega_1}2, <_{\aleph_0})$. 
	
	In \cite{koszmidersets} Koszmider was able to answer Question~\ref{hajnalsetsquestion} affirmatively.
	\begin{theorem}[\cite{koszmidersets}]\label{koszmidersettheorem}
		Assuming $\square_\omo$, there is a ccc forcing which adds an $\omt$ length chain in $(\fsp{\omo}{2}, <_\azero)$ which is to say, adds a sequence of sets $\seq{X_\alpha \mid \alpha < \omt}$ in $[\omo]^\aone$ such that for $\alpha < \beta< \omt$, $X_\alpha \setminus X_\beta$ is finite, and $|X_\beta \setminus X_\alpha| = \aone$. In particular, it is consistent that there is an $\omt$ length chain in $(\omobaire, <_{\azero})$.
	\end{theorem}
	We do not give the definition of $\square_\omo$ here as we will not require it in what follows. Later, Koszmider was able to answer Question~\ref{hajnalquestion} as well. 
	\begin{theorem}[\cite{koszmiderfunctions}] \label{koszmiderfunctheorem}
		It is consistent that  $(\omobaire, \ll_{\azero})$ contains an $\omt$ length chain.
	\end{theorem}
	We point out that Mitchell \cite{mitchell}, and Veli{\v c}kovi{\'c} \cite{velickovic} following Mitchell, have given alternate presentations of the result of Koszmider using the method of `models of two types' (see \cite{neeman}). The former starting from $\gu$, the latter assuming the negation of a weak form of Chang's Conjecture (we do not define this exact weak form as we will not require it, but in Definition~\ref{defnchang} we define Chang's Conjecture).
	
	In contrast to the results of Baumgartner, the results of Koszmider are specific to the ideal $[\aone]^{<\azero}$, so in particular an ideal of the form $[\kappa]^{<\theta}$ when $\theta^+ =\kappa$. Is it possible to prove a result as general as Baumgartner's for chains without the assumption of $\theta^+ =\kappa$? 
	
	All the results in this paper are negative answers to questions in this spirit. Our starting point was a result of Shelah which implies that one cannot have $\ll_\theta$-chains in situations where the size of the gap between $\theta$ and $\kappa$ is finite but larger than one. Focusing on the smallest case where Shelah's theorem applies we have the following. 
	\begin{theorem}[\cite{shelah}] \label{introshelah}
		For every infinite cardinal $\theta$, there are no chains in $\left({}^{\theta^{++}}{\theta^{++}}, \ll_{\theta}\right)$ of length $\theta^{+++}$.
	\end{theorem}
	We strengthen this result to $<_\theta$-chains and hence to sets.
	\begin{thma} 
		Let $\theta$ be an infinite cardinal.
		\begin{enumerate}
			\item There are no chains in $(\fsp{\theta^{++}}{\theta^{++}}, <_{\theta})$ of length $\theta^{+++}$.
			\item In particular, there is no sequence $\langle S_\alpha\mid\alpha<\theta^{+++}\rangle$ of sets in $[\theta^{++}]^{\theta^{++}}$ 
			such that, for all $\alpha<\beta<\theta^{+++}$:
			\begin{enumerate}
				\item $|S_\alpha\setminus S_\beta|<\theta$;
				\item $|S_\beta\setminus S_\alpha|=\theta^{++}$. 
			\end{enumerate}
		\end{enumerate}
	\end{thma}
	
	Taking $\theta= \aleph_0$ this tells us that there cannot be an $\omthr$ chain of sets in $[\omt]^\atwo$ strongly increasing modulo finite. Shelah's result implies that there are no $\omthr$ chains of functions in $(\fsp{\omt}{\omt}, \ll_\azero)$. 
	
	We prove some related results. As we have hinted above, one can phrase a general form of the question of Hajnal and Szentmikl\'{o}ssy.
	\begin{question}\label{generalhajnal}
		Let $\theta< \kappa$ be infinite cardinals. What are the possible lengths of chains in $({}^\kappa\kappa, \ll_\theta)$?
	\end{question}
	In response, we have the following.
	\begin{thmb}
		Let $\theta< \kappa$ be infinite cardinals with $\kappa$ a limit cardinal but not a cardinal fixed point. Then there are no chains in $({}^\kappa\kappa, \ll_\theta)$ of length $\kappa^+$.
	\end{thmb}
	We point out that the hypotheses in this theorem imply that in fact $\theta^+ < \kappa$, and this is crucial. Except for one case each in Theorem~\ref{weakchainone} and Theorem~\ref{weakchaintwo}, none of the results proved in this article apply to the situation when $\theta^+ =\kappa$, and about these two exceptions see the discussion of Theorem~\ref{cctheorem}.
	
	Shelah has made the following conjecture.
	\begin{conjecture}[\cite{shelah}] \label{shelahconj} Let $\theta$ be an infinite successor cardinal. For every $\mu \geq \theta$, there are no chains in $({}^{\theta^{++}}\mu, \ll_{\theta})$ of length $\mu^+$.
	\end{conjecture}
	Here our progress is quite meagre and in particular does not apply to the situation from PCF Theory which possibly motivated Shelah, when $\mu$ is the limit of an increasing $\theta^{++}$-sequence of regular cardinals greater than $\theta$. We remind the reader that for $\theta$ a cardinal and $\nu$ an ordinal, $\theta^{+\nu}$ denotes the $\nu$th cardinal after $\theta$. That is, if $\theta = \aleph_\xi$, then $\theta^{+\nu}: = \aleph_{\xi+\nu}$.
	\begin{thmc}
		Let $\theta \leq \mu$ be infinite cardinals with $\mu< \theta^{+{(\theta^{++})}}$. Then there are no chains in $({}^{\theta^{++}}\mu, \ll_{\theta})$ of length $\mu^+$.
	\end{thmc}
	
	Apart from the remaining cases of Question~\ref{generalhajnal} and Conjecture~\ref{shelahconj}, whether Koszmider's results can be improved upon by constructing longer chains in $\left({}^{\omega_1}2, <_\azero\right)$ and $(\omobaire, \ll_{\azero})$ remains open as well. In \cite{zapletal} and \cite{shelah} one may find more open problems. Two other papers on topics related to this paper which we have not mentioned are \cite{laver} and \cite{koszmider2017constructions}. In the former, Laver shows, among other things, why one cannot approach the question of Hajnal and Szentmikl\'{o}ssy in the usual way as in the Baire space, which is to say, by iteratively adding functions dominating modulo finite. In the latter, Koszmider studies the morasses which he used to prove Theorem~\ref{koszmiderfunctheorem} in greater detail.

	We now describe the organisation of the paper. In Section~\ref{sectionprep} we collect the tools we will use throughout the paper. In Section~\ref{sectionweakchains} we shall focus on the space $({}^\kappa\mu, <_\theta)$ and prove Theorem A. In Section~\ref{sectionstrongchains} we focus on the space $({}^\kappa\mu, \ll_\theta)$ and prove Theorem B and Theorem C. 
	
	\section{Preparatory lemmata}\label{sectionprep}
	Our notation is standard for set theory. We shall use the Greek letters $\theta$, $\kappa$, $\mu$, $\lambda$, $\varkappa$, $\rho$, $\chi$, $\vartheta$ for cardinals, of which we remind the reader that the letters $\varkappa$ and $\vartheta$ are, respectively, variants of $\kappa$ and $\theta$. Except for $\mu$ which may occasionally be finite, all the others will always denote infinite cardinals, and of these $\lambda$ will always denote an infinite regular cardinal. We shall use the Greek letter $\pi$, possibly with subscripts, for bijections, and all other Greek letters (save for $\omega$) as well as $i$ and $j$ for ordinals, possibly finite. 
	
	For $f$ a function and $X$ a subset of its domain, $f \restriction X$ denotes the function obtained by restricting the domain of $f$ to $X$. If $f, f'$ are two functions with the same domain, when we say that $f \leq f'$ \emph{everywhere} we mean that for every element $x$ of their common domain, $f(x) \leq f'(x)$.
	\subsection{Cardinal invariants} \label{subsectioncardinals}
	We introduce some cardinal invariants we shall use in what follows. 
	\begin{definition}
		Let $\theta, \chi,\rho, \kappa$ be infinite cardinals with $\theta < \chi$ and $\theta \leq \rho$ and $\theta \leq \kappa$.
		Let $\meet(\kappa,  \chi ,\rho, \theta)$ denote the least size of a family $\mathcal B$ in $[\kappa]^{<\chi}$ such that for every $A \in [\kappa]^{\geq\rho}$ there is some $B \in \mathcal B$ such that $|B \cap A|\geq \theta$. As important special cases, let $\meet^-(\kappa, \theta)$ denote $\meet(\kappa,  \theta^+, \kappa, \theta)$, let $\meet(\kappa, \rho, \theta)$ denote $\meet(\kappa, \theta^+,\rho, \theta)$, and let $\meet(\kappa, \theta)$ denote $\meet(\kappa, \theta^+, \theta, \theta)$.
	\end{definition}
	\begin{definition} \label{defncof}
		Let $\theta$ and $\kappa$ be infinite cardinals with $\theta \leq \kappa$. Let $\cof([\kappa]^\theta)$ denote the least size of a family $\mathcal B$ in $ [\kappa]^\theta$
		such that for every $A\in[\kappa]^\theta$ there is some $B\in\mathcal B$ with $A\s B$.
	\end{definition}
	The cardinal $\cof([\kappa]^\theta)$ is the cofinality of the partially ordered set $([\kappa]^\theta, \s)$. It is also sometimes in the context of PCF theory called $\cov(\kappa,\theta^+, \theta^+, 2)$. One may find a detailed treatment of it in \cite[Chapter II, \S5]{shelah94}. The cardinal $\meet(\kappa, \chi , \rho,\theta)$ is somewhat less common: in the notation of \cite{she513}, $\meet(\kappa, \chi^+, \rho, \theta)= \ecf_{\rho, \chi, \theta}(\kappa, \chi)$ (see \cite[Claim~4.5(2)]{she513}); particular instances of $\meet(\kappa, \chi , \rho,\theta)$ appear, for example, in \cite[Chapter~IX,\S5]{shelah94}; see also the cardinal $\mathbf U_J[\mu]$ from (among others) \cite{she775}. The cardinal $\meet(\kappa, \theta)$ is more frequently found, especially the instances of it which follow from Shelah's Revised GCH (\cite{she460})---see for example \cite[Theorem~8.21]{abrahammagidor} and \cite[Theorem~8]{cummingsshelah}. In \cite{lambiehanson} a Galvin-Hajnal theorem for $\meet(\kappa, \theta)$ is proved. In our notation we have followed \cite{matet} where the invariant $\meet(\kappa, \theta)$ was considered (though we order the parameters so as to follow Shelah's convention for the covering numbers from \cite{shelah94}). In particular, the letter `m' stands for `meeting'. In \cite{saf51} a related cardinal invariant for `meeting' functions was considered.
	
	We calculate the values of these cardinal invariants in some simple circumstances. We were informed of the following lemma by Assaf Rinot. Below, when we talk about a family `witnessing' some instance of these cardinal invariants, we shall mean the following: in case the value of the invariant or an upper bound for it is known beforehand (though not necessarily specified), then the family is chosen so that it has size equal to this invariant; in case the value of the invariant is not known (so that we are in the process of obtaining a bound on its size), then we shall mean that the family satisfies the properties required to be a candidate for witnessing the invariant, so that the value of the invariant is at most the size of this family.
	
	Recall that for $\theta$ a cardinal and $\nu$ an ordinal, $\theta^{+\nu}$ denotes the $\nu$th cardinal after $\theta$. That is, if $\theta = \aleph_\xi$, then $\theta^{+\nu}: = \aleph_{\xi+\nu}$. 
	\begin{lemma} \label{pumpup}Let $\theta<\mu<\kappa$ be infinite cardinals.
		\begin{enumerate}
			\item  $\meet(\kappa', \chi',\rho', \theta') \geq \meet(\kappa, \chi ,\rho, \theta)$ for infinite cardinals $\kappa' \geq \kappa$, $\chi' \leq \chi$, $\rho' \leq \rho$,  and $\theta' \geq \theta$.
			\item $\meet^-(\theta^+,\theta)=\meet(\theta^+, \theta) = \theta^+$.
			\item $\kappa \leq \meet^-(\kappa, \theta) \leq \meet(\kappa ,\theta) \leq \cof([\kappa]^\theta)$.
			\item Let $\chi$ be a cardinal such that $\meet^-(\mu, \theta) \leq \chi$ and $\meet^-(\kappa, \mu) \leq \chi$. Then $\meet^-(\kappa, \theta) \leq \chi$.
			\item If $\rho \in [\theta, \mu)\cap \card$, then $\meet(\mu^+,\rho, \theta) \leq \mu^+ \cdot \meet(\mu, \rho, \theta)$.
			\item If $\rho \in [\theta, \kappa)\cap \card$ and $\kappa < \theta^{+{\cf(\rho)}}$, then $\meet(\kappa, \rho, \theta) = \kappa$. In particular, if $\kappa< \theta^{+\cf(\theta)}$, then $\meet(\kappa, \theta) = \kappa$.
			\item If $\kappa$ is a limit cardinal and $\meet(\chi, \theta) \leq \kappa$ for a tail of $\chi \in(\theta, \kappa)\cap \card$, then $\meet^-(\kappa, \theta) = \kappa$.
			\item If there are no cardinal fixed points in $(\theta,\kappa]$, then $\meet^-(\kappa, \theta) = \kappa$.
			\item If $\kappa$ is uncountable and not a cardinal fixed point then there is some $\chi < \kappa$ such that $\meet^-(\kappa, \chi) = \kappa$.
			\item $\cof([\theta^+]^\theta) = \theta^+$.
			\item Let $\chi$ be a cardinal such that $\cof([\mu]^\theta) \leq \chi$ and $\cof([\kappa]^\mu) \leq \chi$ then $\cof([\kappa]^\theta) \leq \chi$.
			\item If $\kappa<\theta^{+\omega}$, then $\cof([\kappa]^\theta)=\kappa$.
			\item Let $\chi: =\kappa^{<\theta}$. Then  $\cof([\kappa]^\theta) \leq \meet(\chi, \theta)$.
		\end{enumerate}
	\end{lemma}
	\begin{proof}
		\begin{enumerate}
			\item Clear from the definitions.
			\item Since $\theta^+$ itself witnesses $\meet^-(\theta^+,\theta)$ and $\meet(\theta^+,\theta)$. 
			\item The first inequality is the only one which does not follow automatically from the definitions. So let $\langle S_i \mid i< \kappa\rangle$ be pairwise disjoint elements of $[\kappa]^\kappa$. Now if $\mathcal B$ witnesses $\meet^-(\kappa, \theta)$ and $|\mathcal B| = \chi < \kappa$ then we get a contradiction since $\chi \cdot\theta < \kappa$ so for some $i< \kappa$, $S_i$ has empty intersection with every element of $\mathcal B$.
			\item Let $\{B_\xi \mid \xi < \chi\}$ witness $\meet^-(\mu, \theta) \leq \chi$ and let $\{C_\nu \mid \nu < \chi\}$ witness $\meet^-(\kappa, \mu)\leq \chi$. For each $\nu < \chi$ let $\pi_\nu: \mu \leftrightarrow C_\nu$ be a bijection. Then $\{\pi_\nu``[B_\xi] \mid \nu < \chi, \xi < \chi\}$ witnesses $\meet^-(\kappa, \theta) \leq \chi$.
			\item Since $\mu^+$ is regular, any element of $[\mu^+]^\rho$ is bounded in $\mu^+$. So let $\chi:=\meet(\mu,\rho, \theta)$ and let $\{B_\xi \mid \xi < \chi\}$ witness this and for every $\nu \in [\mu, \mu^+)$ let $\pi_\nu: \mu \leftrightarrow \nu$ be a bijection. Then $\{\pi_\nu[B_\xi] \mid \nu \in[\mu, \mu^+), \xi < \chi\}$ witnesses $\meet(\mu^+,\rho, \theta)\leq \mu^+\cdot \chi$ by our observation. 
			\item We prove this by induction on $\kappa$. The base case is clear by (ii) and the successor case is clear by (v). For the limit case, as $\kappa< \theta^{+\cf(\rho)}$, we have that $\cf(\kappa) < \cf(\rho)$. It follows that if $A \in[\kappa]^\rho$, then for some $\chi < \kappa$, $|A \cap \chi| \geq \rho$. So now by the induction hypothesis for each $\chi \in (\theta, \kappa) \cap \card$, let $\{B^\chi_\xi \mid \xi< \chi\}$ witness that $\meet(\chi, \rho,\theta) = \chi$. So we then have that $\{B^\chi_\xi \mid \chi \in (\theta, \kappa) \cap \card, \xi< \chi\}$ witnesses that $\meet(\kappa,\rho, \theta) = \kappa$ by the observation.
			\item Let $\mathfrak a$ be an end segment of $(\theta, \kappa) \cap \card$ such that for every $\chi \in \mathfrak a$, $\meet(\chi, 
			\theta) \leq \kappa$ as witnessed by some $\{B^\chi_\xi \mid \xi< \kappa\}$. Let us verify that $\{B^\chi_\xi \mid \chi \in \mathfrak a, \xi< \kappa\}$ witnesses that $\meet^-(\kappa, \theta) = \kappa$. So let $A \in [\kappa]^\kappa$. Then there is some $\chi \in \mathfrak a$ such that $|A \cap \chi| \geq \theta$. So there is some $\xi< \kappa$ such that $|B^\chi_\xi \cap A \cap \chi| \geq \theta$.
			\item We shall prove this by induction on $\kappa$. The base case is clear by (ii) and the successor case is clear by (iv). So suppose that $\kappa$ is a limit cardinal, which is not a cardinal fixed point by our assumptions. Suppose that $\kappa = \aleph_\xi > \xi$. Let $\chi  \in (\theta, \kappa) \cap \card$ be a regular cardinal such that $\xi< \chi$. So then $\chi < \kappa < \chi^{+\chi}$. So for all $\rho \in (\chi, \kappa) \cap \card$ we have by (vi) that $\meet(\rho, \chi) = \rho$. So, by (vii) we have that $\meet^-(\kappa, \chi) = \kappa$. Now, the induction hypothesis tells us that $\meet^-(\chi, \theta) = \chi$. So we use (iii) and (iv) to conclude that $\meet^-(\kappa, \theta) = \kappa$ and continue the induction.
			\item If $\kappa$ is not a cardinal fixed point, then it is not a limit of cardinal fixed points either. So we can find a $\chi < \kappa$ such that there are no cardinal fixed points in $(\chi, \kappa]$. Then $\meet^-(\kappa, \chi) = \kappa$ by (viii).
			\item Again, $\theta^+$ serves as a witness.
			\item Let $\{B_\xi \mid \xi < \chi\}$ witness $\cof([\mu]^\theta) \leq \chi$ and let $\{C_\nu \mid \nu < \chi\}$ witness $\cof([\kappa]^\mu) \leq \chi$. For each $\nu < \chi$ let $\pi_\nu: \mu \leftrightarrow C_\nu$ be a bijection. Then $\{\pi_\nu``[B_\xi] \mid \nu < \chi, \xi < \chi\}$ witnesses $\cof([\kappa]^\theta) \leq \chi$.
			\item By induction using (iii) and (x) and (xi).
			\item Let $\pi:{}^{<\theta}\kappa \rightarrow \chi$ be a bijection. Let $\mathcal B$ be a family witnessing $\meet(\chi, \theta)$. Define
			$$\mathcal C:= \{\bigcup \im[\pi^{-1}[B]] \mid B \in \mathcal B\}.$$
			We claim that $\mathcal C$ witnesses $\cof([\kappa]^\theta)$. It is easy to see that $\mathcal C \s [\kappa]^\theta$. To verify the rest, suppose that $A \in [\kappa]^\theta$. Let $h: \theta \rightarrow \kappa$ be a function such that $\im(h) = A$. Let $Y:= \{\pi(h\restriction \xi) \mid \xi < \theta\}$.
			So then $ Y\in [\chi]^\theta$. So there is some $B \in \mathcal B $ such that $|B \cap Y| \geq \theta$. It follows that 
			$$Z:=\{\xi < \theta \mid h\restriction \xi \in \pi^{-1}[B]\}= \{\xi < \theta \mid \pi(h\restriction \xi) \in B\}$$ has size $\theta$, and in particular is unbounded in $\theta$. Now supposing that $\gamma \in A$, then there is some $\xi \in Z$ such that $\gamma \in \im(h \restriction \xi)$. It follows that $A \s \bigcup\im[\pi^{-1}[B]]$.\qedhere
		\end{enumerate}
	\end{proof}
	
	Beyond Lemma~\ref{pumpup} and the other results we have hinted at, the following result connects the cardinals $\meet(\kappa, \theta)$ and $\cof([\kappa]^\theta)$ to a standard hypothesis from PCF Theory. Below, the equivalence between (i) and (ii) is due to Shelah \cite[Theorem 6.3]{she420} and their equivalence with (iii) is due to Rinot \cite{saf05}. The reader may consult either of these papers for undefined notions. 
	\begin{theorem}[\cite{she420,saf05}]
		The following are equivalent:
		\begin{enumerate}
			\item Shelah's Strong Hypothesis, that is, $\pp(\kappa) = \kappa^+$ for every singular cardinal $\kappa$;
			\item for every pair of cardinals $\theta< \kappa$, $\cof([\kappa]^\theta) = \kappa$ if $\cf(\theta)> \cf(\kappa)$ and $\kappa^+$ otherwise;
			\item  for every pair of cardinals $\theta< \kappa$, $\meet(\kappa, \theta)= \kappa$ if $\cf(\theta) \neq \cf(\kappa)$ and $\kappa^+$ otherwise.
		\end{enumerate}
	\end{theorem}
	One may extract further bounds on these cardinals in terms of PCF-theoretic invariants from the literature, see for example the covering numbers in \cite[Chapter II, \S5]{shelah94} and \cite[Claim 1.2]{Sh:430}.
	Lastly, Magidor proved in \cite{magidor} that assuming the consistency of certain large cardinals, it is consistent that $\cof([\aleph_\omega]^{\aleph_0})> \aleph_{\omega+1}$, and Shelah proved in \cite{she137} that for every countable ordinal $\alpha$, it is consistent assuming the consistency of large cardinals that $\cof([\aleph_\omega]^{\aleph_0})> \aleph_{\alpha}$. It is known that large cardinals are required. See \cite{gitikmagidor} for details.
	
	\subsection{Convex equivalence relations}
	We shall say that an equivalence relation $\equiv$ on a linear order $L$ is \emph{convex} if its equivalence classes are convex: if $x <_L y <_L z$ are elements of $L$ and $x \equiv z$, then in fact $x\equiv y \equiv z$. Convex equivalence relations of the sort which appear in Lemma~\ref{convexexample} are the core idea in \cite{shelah} and here as well. 
	
	The following lemma is easy to prove and often occurs in PCF theory, but we give a proof anyway. Recall that for $E$ a set of ordinals, $\cl(E)$ denotes the closure of $E$ in $\sup(E)$, which is to say, $\cl(E): = E \cup \{\alpha < \sup(E) \mid \alpha = \sup(E \cap \alpha)\}$.
	\begin{lemma} \label{convexlemma}
		Let $\lambda$ be an uncountable regular cardinal.
		\begin{enumerate}
			\item If $\equiv$ is a convex equivalence relation on $\lambda$ and there is some $E \s \lambda$ unbounded in $\lambda$ and consisting of pairwise inequivalent elements (that is, ${[E]^2} \cap {\equiv} = \emptyset$), then there is some $ D\s \lambda$ which is a club consisting of pairwise inequivalent elements: ${[D]^2} \cap {\equiv} = \emptyset$.
			\item If $\equiv$ is a convex equivalence relation on $\lambda$, then there is $D \s \lambda$ a club such that either ${[D]^2} \s {\equiv}$ or ${[D]^2} \cap {\equiv} = \emptyset$.
			\item If $\mathcal E$ is a family of convex equivalence relations on $\lambda$ of cardinality less than $\lambda$, then there is $D\s \lambda$ a club such that for every ${\equiv} \in {\mathcal E}$ either ${[D]^2} \s {\equiv}$ or ${[D]^2} \cap {\equiv} = \emptyset$.
		\end{enumerate}
	\end{lemma}
	\begin{proof}
		Assuming that we have proved (i), the proof of (ii) is easy: given $\equiv$ a convex equivalence relation on $\lambda$, in case there is some $E \s \lambda$ an unbounded set consisting of pairwise inequivalent elements, then there is a club consisting of pairwise inequivalent elements. If there is no unbounded subset of $\lambda$ consisting of inequivalent elements, then there must be an equivalence class of $\equiv$ which is an end segment of $\lambda$, and then this end segment is of course a club as well. Obtaining (iii) from (ii) is clear. So we only need to prove (i).
		
		So, suppose that $E \s \lambda$ is unbounded in $\lambda$ and consists of pairwise inequivalent elements. We can assume that if $\alpha \in E$, then for every $\beta < \alpha$, $\beta \not \equiv \alpha$. We shall show that $\cl(E)$ will be a club of $\lambda$ as required. So let, if possible, $\alpha \in \cl(E) \setminus E$. If $\beta \in E \cap \alpha$, then there is some $\gamma \in (\beta, \alpha) \cap E$. By the choice of $E$, $\beta \not \equiv \gamma$. So, as $\equiv$ is convex, for every $\beta < \alpha$, $\beta \not \equiv \alpha$. This implies that $\alpha$ is not equivalent to any element of $E \cap \alpha$. Now if $\beta \in E\setminus \alpha = E\setminus (\alpha+1)$, then by the choice of $E$, $\alpha \not \equiv \beta$. So $\cl(E)$ is a club as required.
	\end{proof}
	We note that the proof of (i) in fact tells us that if a convex equivalence relation $\equiv$ on a regular cardinal $\lambda$ has unboundedly many equivalence classes, then the natural `selector' for $\lambda/{\equiv}$, the set of least elements of each equivalence class, is in fact a club. We shall refer to a club with a property as in (iii) as having the \emph{0-1 property} with respect to the family of convex equivalence relations $\mathcal E$.
	
	We come to the point at which we define the specific convex equivalence relations we shall use. Before this, let us mention that in all that follows in this article we shall have $\vec f=\langle f_\alpha\mid\alpha<\lambda\rangle$ a $<_\theta$-chain in $\fsp{\kappa}{\mu}$, which is to say that for all $\alpha<\beta<\lambda$,
	\begin{enumerate}
		\item $|\{ \xi<\kappa\mid f_\alpha(\xi)>f_\beta(\xi)\}|<\theta$;
		\item $|\{ \xi<\kappa\mid f_\alpha(\xi)<f_\beta(\xi)\}|=\kappa$.
	\end{enumerate}
	At some points $\vec f$ will have some extra properties as well, but what follows in this section is based on the above assumption on $\vec f$. In particular, the convex equivalence relations we define will be on $\lambda$, which we remind the reader by our convention is always a regular cardinal.
	\begin{defn}
		Let $g \in {}^\kappa\mu$ and $Y \s \mu$. Then $g^Y:\kappa\rightarrow Y\cup\{\mu\}$ is the function defined
		as follows: $$g^Y(\xi):=\min((Y\cup\{\mu\})\setminus g(\xi)).$$
	\end{defn}
	We shall use the following fact several times without mentioning.
	\begin{observation}
		Let $g, g' \in {}^\kappa\mu$ be functions, $X \s \kappa$ and $Y \s \mu$. Suppose that $g\restriction X \leq g'\restriction X$ everywhere. Then $g^Y\restriction X \leq g^Y\restriction X$ everywhere.
	\end{observation}
	
	\begin{defn} Given $X\s\kappa$ and $Y\s\mu$, let $\equiv_X^Y$ denote the following relation on $\lambda$:
		$$\{ (\alpha,\beta)\mid \alpha,\beta<\lambda\ \&\ |\{\xi\in X\mid f^Y_\alpha(\xi)\neq f^Y_\beta(\xi)\}|<\theta\}.$$
	\end{defn}

	\begin{lemma} \label{convexexample}Let $X\s\kappa$ and $Y\s\mu$. Then $\equiv_X^Y$ is a convex equivalence relation.
	\end{lemma}
	\begin{proof} As $\theta$ is an infinite cardinal, it is clear that $\equiv^Y_X$ is an equivalence relation. Suppose $\alpha<\beta<\gamma<\lambda$ and $\alpha \equiv_X^Y \gamma$. 
		Let $x\in[\kappa]^{<\theta}$ be such that for all $\xi\in\kappa\setminus x$,
		$f_\alpha(\xi)\le f_\beta(\xi)\le f_\gamma(\xi)$. As $\alpha\equiv_X^Y\gamma$,
		we can also assume that $x$ is large enough that $\{\xi\in X\mid f^Y_\alpha(\xi)\neq f^Y_\gamma(\xi)\}\s x$.
		Now, for all $\xi\in X\setminus x$, we get that $f^Y_\alpha(\xi)=f_\gamma^Y(\xi)$,
		so that the interval $(f_\alpha(\xi),f_\gamma(\xi))$ has empty intersection with $Y$,
		and hence $f_\alpha^Y(\xi)=f_\beta^Y(\xi)=f_\gamma^Y(\xi)$.
		
		Thus $\{\xi\in X\mid f^Y_\alpha(\xi)\neq f^Y_\beta(\xi)\}$ 
		and $\{\xi\in X\mid f^Y_\beta(\xi)\neq f^Y_\gamma(\xi)\}$ are both subsets of $x$.
		In particular, $\alpha \equiv_X^Y \beta$ and $\beta\equiv_X^Y \gamma$.
	\end{proof}
	In what follows, we shall pick $\mathcal B$ a collection of subsets of $\kappa$ and $\mathcal C$ a collection of subsets of $\mu$, and then $\{\equiv^C_B \mid B \in \mathcal B, C \in \mathcal C\}$ will be the family of convex equivalence families of interest to us.
	We end this section with a simple observation. 
	\begin{observation}
		Let $\chi, \rho$ be infinite cardinals. Let $X \in [\kappa]^\chi$ and $Y \in [\mu]^\rho$ be such that $\equiv^Y_X$ has $\lambda$-many equivalence classes. Then $\chi^\rho \geq \lambda$. 
	\end{observation}
	\section{\texorpdfstring{The space $({}^\kappa\mu, <_\theta)$}{The first space}}\label{sectionweakchains}
	
	\begin{definition}\label{defnchang}
		Let $\lambda, \mu, \kappa$ be ordinals.
		\begin{enumerate} 
			\item The partition relation $\lambda \rightarrow [\mu]^2_{\kappa, \mathrm{bdd}}$ asserts that for every $c:[\lambda]^2 \func \kappa$, there is some $I \in [\kappa]^\mu$ and $\epsilon < \kappa$ such that $c``[I]^2 \s \epsilon$. 
			\item Recall that \emph{Chang's Conjecture}, denoted $\cc$, states that $\omega_2\rightarrow[\omega_1]^2_{\omo, \mathrm{bdd}}$.
		\end{enumerate}
	\end{definition}
	We point out that our definition of $\cc$ is a little non-standard and that it is a result of Erd\H{o}s and Hajnal \cite[pg.\ 275]{erdoshajnal} (or see \cite[Proposition 8.2]{kanamori}) that the above formulation is equivalent to the more standard definition.
	Before starting with Theorem~\ref{weakchainone} and Theorem~\ref{weakchaintwo}, we mention the following two results, the first due to Koszmider, the second due to Veli{\v c}kovi{\'c}, which they imply as special cases. 
	\begin{theorem}[\cite{koszmiderfunctions}, \cite{velickovic}] \label{cctheorem}
		\begin{enumerate}
			\item $\cc$ implies that there is no chain in $({}^\omo2, <_{\aleph_0})$ (that is, a chain of sets) of length $\omega_2$.
			\item  $\cc$ implies that there is no chain in $({}^\omo\omo, \ll_{\aleph_0})$ of length $\omega_2$.
			
		\end{enumerate}
	\end{theorem}

	\begin{theorem}\label{weakchainone}
		Let $\theta,\kappa, \mu, \lambda$ be infinite cardinals with $\lambda$ being regular. Suppose also that 
		\begin{enumerate}
			\item $\theta \leq \kappa, \mu$ and $\theta \leq \cf(\kappa)$;
			\item $\meet^-(\kappa, \theta) < \lambda$;
			\item $\cof([\mu]^\theta) < \lambda$;
			\item one of the following occurs:
			\begin{enumerate}
				\item $\theta^+ <\cf(\kappa)$, or
				\item $\lambda \rightarrow [\theta^+]^2_{\kappa, \mathrm{bdd}}$.
			\end{enumerate}
		\end{enumerate}
		Then there are no chains in $({}^\kappa\mu, <_\theta)$ of length $\lambda$.
	\end{theorem}
	\begin{proof}
		Note that the hypotheses imply that $\mu, \kappa < \lambda$. Suppose towards a contradiction that $\langle f_\alpha \mid \alpha< \lambda\rangle$ is such a chain. Let $\mathcal B$ be a family witnessing $\meet^-(\kappa, \theta) < \lambda$, and since $\kappa \leq \meet^-(\kappa, \theta)$ we can assume that $\mathcal B$ is closed under taking end segments. Let $\mathcal C$ be a family witnessing that $\cof([\mu]^\theta) < \lambda$. As the family $\{\equiv^C_B \mid B \in \mathcal B, C \in \mathcal C\}$ of convex equivalence relations has size less than $\lambda$, we can find a club $D \s \lambda$ which has the 0-1 property with respect to it by Lemma~\ref{convexlemma}. Let $\langle i_\alpha \mid \alpha< \lambda\rangle$ be the increasing enumeration of $D$. 
		\begin{claim}\label{01claim}
			For any $\epsilon< \kappa$, there are $B \in \mathcal B$ and $C \in \mathcal C$ such that $B \s \kappa \setminus \epsilon$ and all the elements of $D$ are in distinct classes of $\equiv^C_B$.
		\end{claim}
		\begin{why}
			As $f_{i_0} <_\theta f_{i_1}$ the set $S:= \{\xi \in \kappa \setminus \epsilon \mid f_{i_0}(\xi) < f_{i_1}(\xi)\}$ has size $\kappa$, so there is some $B \in \mathcal B$ such that $|B \cap S| \geq \theta$. Since $\mathcal B$ is closed under taking end segments we can also assume that $ B \s \kappa \setminus \epsilon$. Let $C \in \mathcal C$ be such that $f_{i_0}``[B] \s C$. Then $i_0 \not\equiv^C_B i_1$, so by the 0-1 property of $D$ we can finish.
		\end{why}

		As $\theta \leq \cf(\kappa)$ we are justified in defining the following function $h:[\lambda]^2\rightarrow \kappa$ via $h(\alpha, \beta) := \epsilon$ where $\epsilon<\kappa$ is the least $\varepsilon< \kappa$ such that 
		$$\{\xi< \kappa \mid f_{i_\alpha}(\xi) > f_{i_\beta}(\xi)\} \s \varepsilon.$$
		We now have to make a case distinction to finish the proof.
		
		\textbf{\underline{Case 1: $\cf(\kappa)>\theta^{+}$.}} Let $\epsilon:=\sup\{ h(\alpha,\beta)\mid \alpha<\beta<\theta^+\}$ which we note is bounded below $\kappa$. By Claim~\ref{01claim}, we can find $B \in \mathcal B$ and $C \in \mathcal C$ such that $B \s \kappa \setminus \epsilon$ and such that all the elements of $D$ are in distinct classes of $\equiv^C_B$. This implies that for every $\alpha< \beta < \theta^+$, $f_{i_\alpha}^C\restriction B \leq f_{i_{\beta}}^C\restriction B$ everywhere but also, $i_\alpha \not \equiv^C_B i_\beta$.
		
		So for each $\alpha < \theta^+$ we can find $\tau_\alpha \in B$ and $\zeta_\alpha \in C$ such that $$f_{i_\alpha}^C(\tau_\alpha) = \zeta_\alpha < f_{i_{\alpha+1}}^C(\tau_\alpha).$$ 
		As $|B \times C| < \theta^+$, for some $(\tau^*, \zeta^*)\in B \times C$ we can find $\alpha< \beta < \theta^+$ such that $\tau_\alpha = \tau^* = \tau_\beta$ and $\zeta_\alpha = \zeta^*= \zeta_\beta$. But this implies that 
		$$f_{i_\alpha}^C(\tau^*) = \zeta^* < f_{i_{\alpha+1}}^C(\tau^*) \leq f_{i_\beta}(\tau^*) = \zeta^*$$
		which is not possible. 
		
		\textbf{\underline{Case 2: $\lambda \rightarrow [\theta^+]^2_{\kappa, \mathrm{bdd}}$.}} The proof is similar. Suppose towards a contradiction that $\lambda \rightarrow [\theta^+]^2_{\kappa, \mathrm{bdd}}$. Applying it to the function $h:[\lambda]^2 \func \kappa$, we get an $\epsilon< \kappa$ and an $I \in [\lambda]^{\theta^+}$ such that $h``[I]^2 \s \epsilon$. Appealing to Claim~\ref{01claim} we again find $B$ and $C$ as above with $B \s \kappa \setminus \epsilon$.
		
		The rest of the proof is similar, except that we must consider consecutive elements of $I$ in the increasing enumeration instead of consecutive elements of $\theta^+$ as in the previous case. That is, for each $i \in I$, let $i^+:= \min(I \setminus (i+1))$. Then we find $(\tau_i, \zeta_i) \in B \times C$ such that
		$$f_{i}^C(\tau_i) = \zeta_i < f_{i^+}^C(\tau_i).$$
		We reach a contradiction in the same way.
	\end{proof}
	We point out that in the above, we required the hypothesis $\cof([\mu]^ \theta) < \lambda$ whereas in the forthcoming Theorem~\ref{strongchains} we will be able to make do with weaker hypotheses such as $\meet(\mu, \theta)$ so that we suffice ourselves not with covering a set but merely with having large intersection with it. Comparing clauses (vi) and (xii) of Lemma~\ref{pumpup} makes clear the benefit of such a change. The reason we could not be as parsimonious here was that in the proof above we were not able to obtain large subsets of the image as in the forthcoming Claim~\ref{strongchainsclaim}. 
	Indeed, a chain of sets, that is, a chain in $({}^\kappa 2, <_\theta)$ is also a chain in $({}^\kappa \mu, <_\theta)$ for any $\mu \geq 2$. In contrast to this, the forthcoming Proposition~\ref{strongchainsprop} shows that if $\theta, \mu < \kappa$, then there are no chains in $({}^\kappa\mu,\ll_\theta)$ of length $\mu+1$. Finally, if one is interested not in $({}^\kappa \mu, <_\theta)$ but the subspace of it consisting of increasing or even just non-decreasing functions, then one can reduce the hypotheses from those about covering to meeting. However, as these results are not well-motivated we do not include them here.
	
	The difference in the previous theorem and the next theorem are the requirements on $\mu$. In the former, we required that $\theta \leq \mu$ and $\cof([\mu]^\theta) < \lambda$ (by Definition~\ref{defncof} the first of these is required in order for $\cof([\mu]^\theta)$ to be defined). In what follows we replace these hypotheses by $\mu < \cf(\kappa)$.
	\begin{theorem}\label{weakchaintwo}
		Let $\theta,\kappa, \mu, \lambda$ be infinite cardinals with $\lambda$ being regular. Suppose also that 
		\begin{enumerate}
			\item $\theta \leq \cf(\kappa)$;
			\item $\meet^-(\kappa, \theta) < \lambda$;
			\item $\mu < \cf(\kappa)$;
			\item one of the following occurs:
			\begin{enumerate}
				\item $\theta^+ <\cf(\kappa)$, or
				\item $\lambda \rightarrow [\theta^+]^2_{\kappa, \mathrm{bdd}}$.
			\end{enumerate}
		\end{enumerate}
		Then there are no chains in $({}^\kappa\mu, <_\theta)$ of length $\lambda$.
	\end{theorem}
	\begin{proof}
		The proof is almost exactly the same as that of Theorem~\ref{weakchainone} except the definition of $\mathcal C$ and also Claim~\ref{01claim} which needs to be slightly modified. As before, let $\mathcal B$ be a family witnessing $\meet^-(\kappa, \theta) < \lambda$ which is closed under taking end segments. Let $\mathcal C: =[\mu]^1$. The following is the analogue of Claim~\ref{01claim}.
		\begin{claim}
			For any $\epsilon< \kappa$, there are $B \in \mathcal B$ and $C \in \mathcal C$ such that $B \s \kappa \setminus \epsilon$ and all the elements of $D$ are in distinct classes of $\equiv^C_B$.
		\end{claim}
		\begin{why}
			As $f_{i_0} <_\theta f_{i_1}$ the set $S:= \{\xi \in \kappa \setminus \epsilon \mid f_{i_0}(\xi) < f_{i_1}(\xi)\}$ has size $\kappa$. As $\mu < \cf(\kappa)$, there is a $\zeta < \mu$ such that $S^\zeta:=\{\xi \in S \mid f_{i_0}(\xi) = \zeta\}$ has size $\kappa$. So there is some $B \in \mathcal B$ such that $|B \cap S^\zeta| \geq \theta$ and $ B \s \kappa \setminus \epsilon$. Let $C:= \{\zeta\}$. So then $B$ and $C$ are as required.
		\end{why}
		The rest of the proof is exactly the same.
	\end{proof}
	The next two corollaries together prove Theorem~A.
	\begin{corollary}
		There are no $\theta^{+++}$ chains in $(\fsp{\theta^{++}}{\theta^{++}}, <_\theta)$ for every infinite cardinal $\theta$.
	\end{corollary}
	\begin{proof}
		Suppose that there is such a chain. Then by Lemma~\ref{pumpup},  for $(\mu, \kappa, \lambda): = (\theta^{++},\theta^{++},\theta^{+++})$, $$\theta^{+++} = \lambda >\meet^-(\kappa, \theta) = \meet^-(\theta^{++}, \theta) = \theta^{++}$$ and $$\theta^{+++} = \lambda >\cof([\mu]^\theta) = \cof([\theta^{++}]^ \theta) = \theta^{++}.$$ So by Theorem~\ref{weakchainone} we conclude that $\cf(\kappa) = \cf(\theta^{++})=\theta^{++} \leq \theta^+$ which is absurd.
	\end{proof}
	\begin{corollary} \label{setsresult} For every infinite cardinal $\theta$, there is no sequence $\langle S_\alpha\mid\alpha<\theta^{+++}\rangle$ of sets in $[\theta^{++}]^{\theta^{++}}$ 
		such that, for all $\alpha<\beta<\theta^{+++}$:
		\begin{enumerate}
			\item $|S_\alpha\setminus S_\beta|<\theta$;
			\item $|S_\beta\setminus S_\alpha|=\theta^{++}$. 
		\end{enumerate}
	\end{corollary}
	\begin{proof}
		If there were such a sequence, then we would get a $\theta^{+++}$ length chain in $({}^{\theta^{++}}2, <_\theta)$ and hence in $(\fsp{\theta^{++}}{\theta^{++}}, <_\theta)$ as well, which we now know to be impossible. Alternately, we could use Theorem~\ref{weakchaintwo} for the parameters $(\mu, \kappa, \lambda): = (2,\theta^{++},\theta^{+++})$ to get a contradiction. 
	\end{proof}
	
	\section{\texorpdfstring{The space $({}^\kappa\mu, \ll_\theta)$}{The second space}}\label{sectionstrongchains}
	The following proposition shows that when considering chains in $({}^\kappa\mu, \ll_\theta)$, the requirement that $\mu \geq \kappa$ is quite natural. This is in contrast to the space $({}^\kappa\mu, <_\theta)$ for which one of the most important cases is when $\mu = 2$, that is we are considering chains of sets. Indeed, one may compare it with Theorem~\ref{koszmidersettheorem} of Koszmider.
	\begin{proposition}\label{strongchainsprop}
		Suppose that $\theta, \kappa$ are infinite cardinals, and $\mu$ is a cardinal which is possibly finite. Suppose that $\mu< \kappa$ and $\theta< \kappa$. Then there are no chains in $({}^\kappa\mu, \ll_\theta)$ of length $\mu+1$.
	\end{proposition}
	\begin{proof}
		Suppose towards a contradiction that $\langle f_\alpha \mid \alpha\leq \mu\rangle$ is such a chain. Then the following set 
		$$S:=\{\xi <\kappa\mid\exists \alpha, \beta\leq \mu\ [\alpha< \beta\ \&\ f_\alpha(\xi) \geq f_\beta(\xi) ]\}$$
		has size at most $\mu \cdot \theta$, so in particular less than $\kappa$. So now let $\xi \in \kappa \setminus S$. Then the function $\alpha \mapsto f_\alpha(\xi)$ defined for $\alpha\leq\mu$ is strictly increasing in $\alpha$ and has image contained in $\mu$. However the image of this function has ordertype $\mu+1$ which leads to a contradiction.
	\end{proof}
	The following theorem is the main result of this section. As it has a large number of parameters, we mention two important instantiations of the theorem. The first is when $\chi= \rho=\theta^+ $ and $\vartheta = \theta$, this being the content of Corollary~\ref{strongchainsfirstcor}. The second is when $\chi= \rho = \vartheta^+$ for some regular $\vartheta \in (\theta, \kappa)$ such that $\kappa < \vartheta^{+\vartheta}$, assuming such a $\vartheta$ can be found. This is the content of Corollary~\ref{strongchainssecondcor}. 
	\begin{theorem} \label{strongchains}
		Suppose that $\mu, \kappa, \theta, \lambda$ are infinite cardinals with $\theta<\theta^+<\kappa \leq \mu$ and $\lambda$ being regular. Suppose also that there are infinite cardinals $\chi, \rho, \vartheta$ such that 
		\begin{enumerate}
			\item $\theta < \chi, \rho < \kappa$;
			\item $\meet(\kappa, \rho, \vartheta, \theta) < \lambda$;
			\item $\meet(\mu, \chi, \kappa, \vartheta) < \lambda$.
		\end{enumerate}
		Then there are no chains in $({}^{\kappa}\mu, \ll_\theta)$ of length $\lambda$.
	\end{theorem}
	\begin{proof} Let $\varkappa:= \chi+\rho$, so that $\varkappa < \kappa$. Note that the hypotheses imply that $\lambda > \mu, \kappa$.
		Suppose towards a contradiction that $\langle f_\alpha \mid \alpha< \lambda\rangle$ is such a chain. Let $\mathcal S =\langle S_\eta \mid \eta< \kappa\rangle$ be a partition of $\kappa$ into $\kappa$-many pairwise disjoint sets of size $\kappa$. Let $\mathcal B$ be a witness to $\meet(\kappa, \rho, \vartheta, \theta) < \lambda$ and we furthermore can assume that $\mathcal B$ is closed under taking intersections with elements of $\mathcal S$. Let $\mathcal C$ be a witness to $\meet(\mu,  \chi,\kappa, \vartheta) < \lambda$. Then the family of convex equivalence relations $\{\equiv^C_B \mid B \in \mathcal B, C \in \mathcal C\}$ has size less than $\lambda$, so let $D \s \lambda$ be a club with the 0-1 property with respect to it by Lemma~\ref{convexlemma}. Let $\langle i_\alpha \mid \alpha < \lambda\rangle$ be the increasing enumeration of $D$. As $\varkappa< \kappa$, the set 
		$$\{\xi < \kappa\mid \exists \alpha, \beta < \varkappa\ [\alpha< \beta\ \&\ f_{i_\alpha}(\xi) \geq f_{i_\beta}(\xi)]\}$$
		has size less than $\kappa$, so we can fix an $\eta< \kappa$ such that 
		$$S_\eta \cap \{\xi< \kappa \mid \exists \alpha, \beta < \varkappa\ [\alpha< \beta\ \&\ f_{i_\alpha}(\xi) \geq f_{i_\beta}(\xi)]\} = \emptyset.$$
		\begin{claim}\label{strongchainsclaim}
			There are sequences $\langle j(\alpha) \mid \alpha< \kappa\rangle$ and $\langle \xi_\alpha \mid \alpha< \kappa\rangle$ such that
			\begin{enumerate}
				\item $\langle j(\alpha) \mid \alpha< \kappa\rangle$ consists of ordinals less than $\kappa$ and possibly contains repetitions;
				\item $\langle \xi_\alpha \mid \alpha< \kappa\rangle$ are all distinct and contained in $S_\eta$;
				\item for every $\alpha< \kappa$, $f_{i_0}(\xi_\alpha) \leq f_{i_{j(\alpha)}}(\xi_\alpha) < f_{i_{\kappa}}(\xi_\alpha)$;
				\item the function $\alpha \mapsto f_{i_{j(\alpha)}}(\xi_\alpha)$ defined for $\alpha< \kappa$ is injective.
			\end{enumerate}
		\end{claim}
		\begin{why}
			We construct the sequences $\langle j(\alpha) \mid \alpha< \kappa\rangle$ and $\langle \xi_\alpha \mid \alpha< \kappa\rangle$ by recursion on $\alpha< \kappa$. The base case is easy to perform: set $j(0) = 0$, and since $f_{i_0} \ll_\theta f_{i_\kappa}$ and $|S_\eta| = \kappa > \theta$ we can find a $\xi_0\in S_\eta$ such that $f_{i_0}(\xi_0) = f_{i_{j(0)}}(\xi_0)< f_{i_{\kappa}}(\xi_0)$. 
			
			So suppose that $\nu< \kappa$ and we have constructed $\langle j(\alpha) \mid \alpha < \nu\rangle$ and $\langle \xi_\alpha \mid \alpha< \nu\rangle$ which satisfy the requirements up to $\nu$ and let $Y_\nu:= \{f_{i_{j(\alpha)}}(\xi_\alpha) \mid \alpha< \nu\}$ which we note has size $|\nu|$. Let $\gamma:= \otp(Y_\nu)$ and note that $\gamma< \kappa$. 
			Let $Z:=[0, \gamma+1) \cup\{\kappa\}$. As $|Z| < \kappa$, the set $S^*$ defined as follows
			$$S^*:= \{\xi \in S_\eta\mid \exists \alpha, \beta \in Z\ [\alpha < \beta\ \&\ f_{i_\alpha}(\xi) \geq f_{i_\beta}(\xi)]\}$$
			has size less than $\kappa$. In particular, we can find a $\xi^* \in S_\eta \setminus (Y_\nu \cup S^*)$.
			
			Now, the function 
			$$\beta \mapsto f_{i_{\beta}}(\xi^*)$$
			defined for $\beta < \gamma+1$ is strictly increasing in $\beta$ and has image contained in $[f_{i_0}(\xi^*), f_{i_{\kappa}}(\xi^*))$ by the choice of $\xi^*$. Since the image has ordertype $\gamma+1$, it follows that for some $\beta < \gamma+1$ we have that $f_{i_{\beta}}(\xi^*) \notin Y_\nu$. So now taking $j(\nu):= \beta$ and $\xi_\nu:= \xi^*$ we can continue the recursive construction.
		\end{why}
		Let $\langle j(\alpha) \mid \alpha < \kappa\rangle$ and $\langle \xi_\alpha \mid \alpha< \kappa\rangle$ be as in the claim and let 
		$$Y:=\{f_{i_{j(\alpha)}}(\xi_\alpha) \mid \alpha< \kappa\}.$$
		Now as $\mathcal C$ witnesses $\meet(\mu, \chi, \kappa, \vartheta)$, there is some $C \in \mathcal C$ such that $|C \cap Y| \geq \vartheta$. Let 
		$$X:=\{\xi_\alpha \mid \alpha< \kappa,\ f_{i_{j(\alpha)}}(\xi_\alpha) \in C \cap Y\}$$ and we note that it is in $[\kappa]^{\geq \vartheta}$. As $\mathcal B$ witnesses $\meet(\kappa,  \rho, \vartheta,\theta)$, there is some $B \in \mathcal B$ such that $|B \cap X| \geq \theta$. As $\mathcal B$ is closed under intersections with $S_\eta$ and $X \s S_\eta$, we can also assume that $B \s S_\eta$.
		\begin{claim}
			$i_0 \not \equiv^C_B i_{\kappa}$.
		\end{claim}
		\begin{why}
			For every $\alpha$ such that $\xi_\alpha \in X$ we have that 
			$$f_{i_0}(\xi_\alpha) \leq f_{i_{j(\alpha)}}(\xi_\alpha) < f_{i_{\kappa}}(\xi_\alpha),$$
			and since $f_{i_{j(\alpha)}}(\xi_\alpha) \in C \cap Y$ this implies that $$f_{i_0}^C(\xi_\alpha) \leq f_{i_{j(\alpha)}}(\xi_\alpha) < f_{i_{\kappa}}^C(\xi_\alpha).$$
			Since $|B \cap X| \geq \theta$ we finish.
		\end{why}
		By the 0-1 property of $D$, this tells us that all the elements of $D$ are inequivalent with respect to $\equiv^C_B$ and in particular this is true of $\{i_\alpha \mid \alpha< \varkappa\}$. As $B \s S_\eta$ we also have that for every $\alpha< \beta< \varkappa$, $f^C_{i_\alpha} \restriction B \leq f^C_{i_\beta} \restriction B$ everywhere.
		
		So, for every $\alpha< \varkappa$ we can find $\tau_\alpha \in B$ and $\zeta_\alpha \in C$ such that $$f^C_{i_\alpha} (\tau_\alpha) = \zeta_\alpha < f^C_{i_{\alpha+1}} (\tau_\alpha).$$
		As $|B| < \rho$ and $|C| < \chi$, we have that $|B\times C| < \rho \cdot \chi \leq \varkappa$. This means that for some $(\tau^*, \zeta^*) \in B \times C$ we can find $\alpha < \beta < \varkappa$ such that $\tau_\alpha = \tau^* = \tau_\beta$ and $\zeta_\alpha = \zeta^* = \zeta_\beta$. But then
		$$f_{i_\alpha}^C(\tau^*) = \zeta^* < f_{i_{\alpha+1}}^C(\tau^*) \leq f_{i_\beta}^C(\tau^*) = \zeta^*.$$
		This cannot be.
	\end{proof}
	\begin{corollary} \label{strongchainsfirstcor}
		Let $\theta < \theta^+ < \kappa$ be infinite cardinals and $\lambda$ an infinite regular cardinal. Suppose that $\meet(\kappa, \theta)< \lambda$. Then there are no chains in $({}^\kappa\kappa, \ll_\theta)$ of length $\lambda$.
	\end{corollary}
	\begin{proof}
		As $\meet^-(\kappa, \theta) \leq \meet(\kappa, \theta) < \lambda$ by Lemma~\ref{pumpup}(iii), we can appeal to Theorem~\ref{strongchains} with $\chi= \rho=\theta^+$ and $\vartheta = \theta$ to obtain the desired conclusion.
	\end{proof}
	We can compare this with the analogous corollary from \cite[Theorem 5]{shelah} where instead of $\meet(\kappa, \theta)< \lambda$ the slightly stronger hypothesis of $\cof([\kappa]^\theta) < \lambda$ would have been used to obtain the same conclusion. It is also clear that proving this corollary using Theorem~\ref{weakchainone} would require further assumptions.
	
	However, glancing at Lemma~\ref{pumpup} and the discussion following it, it is clear that for the most important case of $\lambda = \kappa^+$, assuming the consistency of enough large cardinals there are models of set theory where the hypotheses of Corollary~\ref{strongchainsfirstcor} may not hold even for very small $\kappa$ such as $\aleph_\omega$. The next result, which proves Theorem~B, remedies some of these scenarios.
	\begin{corollary} \label{strongchainssecondcor}Let $\theta < \kappa$ be infinite cardinals where $\kappa$ is a limit cardinal which is not a cardinal fixed point. Then there are no chains in $({}^\kappa\kappa, \ll_\theta)$ of length $\kappa^+$.
	\end{corollary}
	\begin{proof}
		Suppose that $\xi< \kappa$ is such that $\kappa = \aleph_\xi$. Let $\vartheta \in (\theta,\kappa)$ be a regular cardinal such that $\vartheta > \xi$. Then $\vartheta<\kappa < \vartheta^{+\vartheta}$. So by Lemma~\ref{pumpup}(vi) and Lemma~\ref{pumpup}(iii) we have that 
		$$\kappa  =\meet(\kappa, \vartheta) = \meet^-(\kappa, \vartheta)=\meet(\kappa,  \vartheta^+,\kappa, \vartheta).$$ Also, as $\vartheta > \theta$, by Lemma~\ref{pumpup}(i) we have that 
		$$\kappa =\meet(\kappa, \vartheta) = \meet(\kappa, \vartheta^+, \vartheta, \vartheta) = \meet(\kappa, \vartheta^+, \vartheta, \theta).$$
		So we appeal to Theorem~\ref{strongchains} with $\chi= \rho = \vartheta^+$ to finish.
	\end{proof}
	While the previous two corollaries were focused on Question~\ref{generalhajnal}, the next corollary focuses on the most general problem, and the one after, which proves Theorem~C, on Conjecture~\ref{shelahconj}.
	\begin{corollary}
		Let $\theta< \theta^+< \kappa \leq \mu$ be infinite cardinals. Suppose that 
		\begin{enumerate}
			\item $\kappa < \theta^{+\cf(\theta)}$;
			\item $\mu< \theta^{+\cf(\kappa)}$.
		\end{enumerate}
		Then there are no chains in $({}^{\kappa}\mu, \ll_\theta)$ of length $\mu^+$.
	\end{corollary}
	\begin{proof}
		By Lemma~\ref{pumpup}(vi) the hypotheses imply that $\meet(\kappa, \theta) = \kappa$ and $\meet(\mu, \theta^+, \kappa, \theta) = \mu$ hold. So we apply Theorem~\ref{strongchains} with $\vartheta= \theta$ and $\chi=\rho = \theta^+$ to finish.
	\end{proof}
	\begin{corollary}
		Let $\theta\leq \mu<\theta^{+(\theta^{++})}$ be infinite cardinals. Then there are no chains in $({}^{\theta^{++}}\mu, \ll_\theta)$ of length $\mu^+$.
	\end{corollary}
	\begin{proof}
	In case $\mu < \theta^{++}$ then we can appeal to Proposition~\ref{strongchainsprop}. So, we can assume that $\theta^{++}\leq \mu$.	By Lemma~\ref{pumpup}(vi) we have that $\meet(\mu,  \theta^+,\theta^{++}, \theta) = \mu$ and $\meet(\theta^{++}, \theta^+, \theta, \theta) = \theta^{++}$ holds. So we apply Theorem~\ref{strongchains} with $\vartheta= \theta$ and $\chi=\rho = \theta^+$ to finish.
	\end{proof}
	
	\section{Acknowledgements}
	The debt to \cite{shelah} should be clear to any reader who makes it to this point. My interest in the results of Koszmider began when my teachers David Asper{\'o} and Mirna D{\v z}amonja suggested reading \cite{koszmiderfunctions} to start my doctoral studies. Some of this research was done when I was a postdoctoral researcher at the Instituto Tecnol{\'o}gico Aut{\'o}nomo de M{\'e}xico. After listening to a talk on Theorem A in February 2020, Assaf Rinot emphasised isolating the cardinals from Section~\ref{subsectioncardinals} and focusing on the space $({}^\kappa\mu, <_\theta)$ and also informed me of Lemma~\ref{pumpup}. The emphasis on $({}^\kappa\mu, <_\theta)$ makes it clear that the results of Section~\ref{sectionweakchains} are more general than the results in \cite{shelah}. This article has benefited greatly from his suggestions, his notes, and his interest. A conversation with Ido Feldman helped inspire the discovery of Theorem~\ref{strongchains}. While working on the article at Bar-Ilan University I was supported by a postdoctoral fellowship funded by the Israel Science Foundation (grant agreement 2066/18).

\end{document}